\documentclass{amsproc}

\usepackage{amssymb}

\usepackage{graphicx}

\usepackage{tikz}
\usetikzlibrary{matrix}
\usepackage{tikz-cd}

\newtheorem{theorem}{Theorem}[section]
\newtheorem{lemma}[theorem]{Lemma}
\newtheorem{corollary}[theorem]{Corollary}
\newtheorem{proposition}[theorem]{Proposition}

\theoremstyle{definition}
\newtheorem{definition}[theorem]{Definition}
\newtheorem{example}[theorem]{Example}

\theoremstyle{remark}
\newtheorem{remark}[theorem]{Remark}

\numberwithin{equation}{section}

\def\C{{\mathbb C}}
\def\P{{\mathbb P}}
\def\R{{\mathbb R}}

\def\Z{{\mathbb Z}}

\def\Aff{\textrm{Aff}}

\newcommand{\De}{\Delta}

\newcommand{\GL}{\textrm{GL}}

\begin{document}

\title{A Lower Bound of the Hofer-Zehnder Capacity via Delzant Polytopes}


\author{Yichen Liu}
\address{Mathematics Department\\
	University of Illinois at Urbana-Champaign\\Champaign, Illinois 61801}
\curraddr{}
\email{yichen23@illinois.edu}
\thanks{}


\date{}

\begin{abstract}
Given a symplectic toric manifold, the moment maps of sub-circle actions can be modified to be admissible functions in the sense of Hofer-Zehnder. By exploiting the relationship between the period of Hamiltonian sub-circle actions of a symplectic toric manifold and its Delzant polytope, we develop an invariant of Delzant polytopes which gives a lower bound of the Hofer-Zehnder capacity.
\end{abstract}

\maketitle

\section{Introduction}
In 1980s, Gromov~\cite{GR85} proved the celebrated non-squeezing theorem which states that a ball $B^{2n}(r)$ of radius $r$ can be symplectically embedded into a cylinder $B^2(R) \times \R^{2n-2}$ of radius $R$ if and only if $r \leq R$. This led to the notion of the \textit{Gromov width} of a symplectic manifold $(M,\omega)$, which is the radius of the largest ball of the same dimension that can be symplectically embedded into $M$. The Gromov width is the first example of a symplectic capacity. In~\cite{EH89}, Ekeland and Hofer defined symplectic capacities in an axiomatic way and constructed many concrete examples of symplectic capacities. Since then, symplectic capacities play an important role in the study of the symplectic geometry. 

Among symplectic capacities, the Hofer-Zehnder capacity, introduced in \cite{HZ90}, is based on properties of periodic orbits of Hamiltonian flows on $(M, \omega)$. For a smooth function $H$ on $M$, let $X_H$ denote its Hamiltonian vector field. For a compact symplectic manifold $(M,\omega)$, we say $H: M \to \R$ is \textbf{admissible} if 
\begin{itemize}
    \item there exist open sets $U,V$ such that $H|_U = H_{\max}$ and $H|_V = H_{\min}$, and
    \item $X_H$ has no non-constant periodic orbits of period less than $1$.
\end{itemize}
We denote the set of admissible functions by $\mathcal{H}_{ad}(M,\omega) = \{ H \in C^\infty(M)$ \big| $H$ is admissible$\}$. The \textbf{Hofer-Zehnder capacity} of $(M,\omega)$ is defined by \[c_{\textrm{HZ}}(M,\omega) = \textrm{sup} \{H_{\max} -H_{\min} \big| \, H \in \mathcal{H}_{ad}(M, \omega) \}. \]

Recall that we say a circle $S^1 = \R / \Z$ acts on a symplectic manifold $(M,\omega)$ {\bf in a Hamiltonian fashion} if there exists a map $H: M \to \R$ such that $dH = \omega(X, \cdot)$, where $X$ is the fundamental vector field of the $S^1$-action. Such an $H$ is called a {\bf moment map}. Moment maps can be modified to admissible functions (see Lemma~\ref{lemma:modification}). Hence, by studying $S^1$-actions on a symplectic manifold, we can estimate the Hofer-Zehnder capacity using the information from the moment maps. Moreover, with the presence of group actions, one can also study equivariant capacities introduced in \cite{FPP18}. The equivariant technique used in this paper is similar to those in \cite{Density}, \cite{Pelayo06}, where the authors obtained bounds for equivariant capacities. In fact, the authors computed packing densities, which easily determines the equivariant capacities according to \cite[Equation (2)]{FPP18}. For more discussion about equivariant capacities, see \cite{Density}, \cite{FPP18}, \cite{Pelayo06},\cite{Pelayo07},\cite{PS08} and the references therein.

A symplectic toric manifold $(M^{2n},\omega,T^n,\Phi)$ is a connected, closed $2n$-dimensional symplectic manifold equipped with an effective, Hamiltonian $n$-torus action (see Definition~\ref{def:Hamiltonian}). Delzant \cite{De88} classified symplectic toric manifolds by their moment polytopes which are Delzant polytopes (see Definition~\ref{def:Delzant}). These manifolds admit many different $S^1$-actions and the information of actions is encoded by simple combinatorial data, so it is possible to get a reasonably good estimate of the Hofer-Zehnder capacity from the combinatorial data. The main goal of this paper is to describe an explicit algorithm for how to find a lower bound of the Hofer-Zehnder capacity by simply using the information obtained from the Delzant polytope.

In \cite{Lu}, Lu also gave a lower bound of the Hofer-Zehnder capacity of a toric manifold from the combinatorial data of the Delzant polytope. Lu's technique is to find symplectic embeddings of balls whose radius is related to the combinatorial data and use the fact that the Hofer-Zehnder capacity is always bounded from below by the Gromov width, while our proof depends on a careful analysis of how to read the maximal (in terms of the cardinality) stabilizer group of non-fixed points from the combinatorial data.

Let $(M^{2n},\omega,T^n,\Phi)$ be a symplectic toric manifold with the Delzant polytope $\Delta = \Phi(M)$. We can write $\De$ as an intersection of half-spaces
\[\Delta = \bigcap_{i=1}^d \{ x \in \mathbb{R}^n \big| \, \langle x, v_i \rangle \leq \lambda_i \},\]
where $v_i \in \Z^n$ are primitive outward-pointing normal vectors. (Recall that a vector $v \in \Z^n$ is primitive if for any $u \in \Z^n$, the equation $v= nu$ for $n \in \Z$ implies that $n = \pm 1$.)

Let $E$ denote an edge of $\De$. Delzant's classification~\cite{De88} implies that there exists a unique $J_E \subset \{1,2,....,d\}$ with $|J_E| = n-1$ such that
\[E = \Delta \cap \bigcap_{j \in J_E}  \{ x \in \mathbb{R}^n \big| \, \langle x, v_j \rangle = \lambda_j \}.  \]

Let $u \in \Z^n$ be a primitive vector. To each edge $E$ of $\De$, we associate a lattice polytope $P^u_E =\{tu-\sum_{j\in J_E}t_jv_j \big| \, 0\leq t,t_j \leq 1\}$, i.e. $P^u_E$ is the convex hull of $\{-v_j: j \in J_E\} \cup \{u\}$.
Let $k^u_E: = |\Z^n \cap \textrm{int}(P^u_E)|+1$ denote the number of interior lattice points plus one. Notice that if $u$ is a linear combination of $v_j$, then $k^u_E = 1$.

Define \[m_u(\De) := \max_E \{k^u_E \big| \, E \textrm{ is an edge of } \De\},\] 
\[\mathcal{T}_u(\De) := \frac{1}{m_u(\De)} (\smash{\max_{x\in \Delta}} \langle x, u \rangle - \smash{\min_{x\in \Delta}} \langle x, u \rangle).\]

The \textbf{toric width} of a Delzant polytope $\Delta \subset \R^n$ is
\[w_T(\De) := \sup_u \{\mathcal{T}_u(\De) \big| \, u \in \Z^n \textrm{ is a primitive vector} \}.\]

Our main result states that for a symplectic toric manifold, the toric width of its Delzant polytope gives a lower bound of the Hofer-Zehnder capacity.

\begin{theorem}\label{thm:main}
    Let $(M^{2n},\omega,T^n,\Phi)$ be a symplectic toric manifold with a Delzant polytope $\De = \Phi(M)$. Let $w_T(\De)$ be the toric width of $\De$, then
    \[c_{\textrm{HZ}}(M,\omega) \geq w_T(\De). \] 
\end{theorem}
\begin{proof}
    For any primitive vector $u \in \Z^n$, let $S^1_u \subset T^n$ be the circle whose Lie algebra is $\R u$. $S^1_u$ acts on $M^{2n}$ in a Hamiltonian fashion with the moment map $\mu= \langle \Phi,u \rangle$. By Corollary~\ref{cor:stabilizer},
    \[m_u(\De) = \max\{ \, |\textrm{stab}_p|:\, p \in M \setminus M^{S^1_u}\}\]
    
    Since $u$ is primitive, this action is effective. By Lemma~\ref{lemma:modification}, there exists an admissible function $H$ such that
    \[H_{\max} -H_{\min} = \frac{1}{m_u}(\mu_{\max} -\mu_{\min}) =\mathcal{T}_u.\]

    By definition, $c_{\textrm{HZ}}(M,\omega) \geq w_T(\De)$.
\end{proof}

Recall that two toric actions $\rho_1,\rho_2: T^n \times M^{2n} \to M^{2n}$ are {\bf equivalent} if there exist a symplectomorphism $\Psi: M \to M$ and an automorphism $h: T^n \to T^n$ such that the following diagram commutes:

\[ \begin{tikzcd} 
        T^n \times M^{2n} \arrow[r,"\rho_1"] \arrow[d,"{(h,\Psi)}"] 
        & M^{2n} \arrow[d,"f"]\\ T^n \times M^{2n} \arrow[r,"\rho_2"] & M^{2n} 
    \end{tikzcd}  \]

Karshon, Kessler and Pinsonnault~\cite{KKP07} proved that two toric actions are equivalent if and only if their Delzant polytopes differ by an affine transformation. Indeed, the toric width is invariant under the affine transformations, so equivalent toric actions will give the same bound in Theorem~\ref{thm:main}.

\begin{lemma}
    The toric width of a Delzant polytope is invariant under the action of the group of affine transformations $\Aff(n,\R)= \GL(n,\Z) \ltimes \R^n$.
\end{lemma}

\begin{proof}
    For any matrix $M \in \GL(n,\Z)$, let $\De'= M\De$ be the image of $\De$ under $M$.  For any edge $E \subset \De$, let $E'\subset \De'$ be the corresponding edge. 
    
    Notice for any $y \in \Delta'$, $\langle M^{-1}y , v_i \rangle \leq \lambda_i$, i.e. $\langle y, (M^{-1})^{T}v_i \rangle \leq \lambda_i$. Hence, for any primitive vector $u \in \Z^n$ and $u'= M^T u$,
    \[P^u_{E'} =\{tu-\sum_{j\in J_E}t_j (M^{-1})^{T} v_j \big| \, 0\leq t,t_j \leq 1\} = (M^{-1})^{T}P^{u'}_E.\] 
    
    Since $M$ preserves the lattice $\Z^n$, 
    \[k_{E'}^{u}=  |\Z^n \cap \textrm{int}(P^u_{E'})|+1 =   | (M^{-1})^{T}(\Z^n \cap \textrm{int}(P^{u'}_E))|+1 = | \Z^n \cap \textrm{int}(P^{u'}_E)|+1  = k_E^{u'}.\]
    
    Therefore, $m_{u}(\De')  = m_{u'}(\De)$.
   
    Furthermore, since for any $y \in \De'$, there exists a unique $x \in \De$ such that $y= Mx$ and $\langle y, u \rangle  = \langle Mx, u \rangle =  \langle x, M^{T}u \rangle = \langle x, u'\rangle $. It follows that $\mathcal{T}_u(\De') = \mathcal{T}_{u'}(\De)$.
   
    Since $M$ maps the set of primitive vectors bijectively to itself, $w_T(\De) = w_T(\De')$.

    On the other hand, translations change $\smash{\max_{x\in \Delta}} \langle x, u \rangle$ and $\smash{\min_{x\in \Delta}} \langle x, u \rangle$ by the same amount while leaving the normal vectors $v_j$ and thus $k^u_E$ unchanged. Hence, translations do not change the toric width.
    
    Therefore, $w_T(\De) = w_T(M\De+v)$ for any $M \in \GL(n,\Z)$ and any $v \in \R^n$.
\end{proof}

However, a symplectic manifold could admit many inequivalent toric actions. For instance, Karshon, Kessler and Pinsonnault~\cite{KKP07} proved that each compact symplectic $4$-manifold admits finitely many inequivalent toric actions, so we immediately have the following corollary.
\begin{corollary}
    Let $(M,\omega)$ be a symplectic manifold that admits a toric action. Then
    \[c_{\textrm{HZ}}(M,\omega) \geq  \sup_{\De} \{w_T(\De) | \De \textrm{ is a Delzant polytope of a toric action.}\}. \] 
\end{corollary}

In Section~\ref{s2}, we review the basic notions of symplectic toric manifolds and study the sub-circle actions. In Section~\ref{s3}, we prove the key lemma that the moment map of a Hamiltonian circle action can be modified to an admissible function. In Section~\ref{s4}, we give an explicit formula for the lower bound of the Hofer-Zehnder capacity of $4$-dimensional symplectic toric manifolds and use the formula to find lower bounds of some familiar manifolds. Furthermore, we compare the lower bound with the result of \cite{HS17} to show our estimate is sharp in many cases.

\subsubsection*{Convention}
Throughout this paper, if not otherwise specified, we will always assume the symplectic manifold $(M,\omega)$ in discussion is compact and connected. For the Hofer-Zehnder capacity of non-compact symplectic manifold, one needs to further require the admissible functions to be compactly supported.

\subsubsection*{Acknowledgements}
This paper is the result of an REU project at the University of Michigan, Ann Arbor, during the summer semester, 2019. I would like to thank Professor Alejandro Uribe for teaching me symplectic geometry, for introducing this problem to me, and for his invaluable help with many aspects of this project.  I am very grateful to the Department of Mathematics at the University of Michigan for supporting this project. I also want to thank Jun Li for encouraging me to compile my results into this paper and Shengzhen Ning for pointing out a mistake in an earlier version and suggesting a way to fix it.

\section{Sub-circle Actions on Symplectic Toric Manifolds}\label{s2}
In this section, we recall some facts about symplectic toric manifolds and study sub-cirlce actions. Throughout this section, we will use the convention that $T = \R^n/ \Z^n$ and its Lie algebra is $\mathfrak{t} \cong \R^n$. Furthermore, we fix an inner product on $\mathfrak{t}$ that identifies $\mathfrak{t}$ with $\mathfrak{t}^*$. 

\subsection{Symplectic Toric Manifolds}

\begin{definition}\label{def:Hamiltonian}
    Suppose $T$ acts on $(M, \omega)$ via symplectomorphisms. This action is \textbf{Hamiltonian} if there exists a $T$-invariant map $\Phi: M \rightarrow \mathfrak{t}^* \cong \R^n$ such that for each $X \in \mathfrak{t}$, 
    \[d\Phi^X = \omega(X^\#, \cdot)\] 
    where $\Phi^X(p): = \langle \Phi(p), X \rangle$, and $X^\#$ is the fundamental vector field corresponding to $X$.

The tuple $(M, \omega, T, \Phi)$ is called a \textbf{Hamiltonian $T$-space} and $\Phi$ is called a \textbf{moment map}.
\end{definition}

For any compact connected Hamiltonian $T$-space, its moment image $\Phi(M)$ is a convex polytope \cite{Atiyah1982,GS82}, called the \textbf{moment polytope}.

Since the orbits of the $T$-action are isotropic submanifolds, if $T$ acts effectively, i.e. every non-identity element acts non-trivially, then $\dim (T) \leq \frac{1}{2} \dim (M)$. For a compact, connected Hamiltonian $T$-space $(M, \omega, T, \Phi)$, if $T$ acts effectively and $\dim (T) = \frac{1}{2} \dim (M)$, we call the tuple $(M, \omega, T, \Phi)$ a \textbf{symplectic toric manifold}. The moment polytope of a symplectic toric manifold satisfies the following definition.

\begin{definition}\label{def:Delzant}
    A \textbf{Delzant polytope} $\Delta \subset \R^n$ is a polytope satisfying:
\begin{itemize}
    \item \textbf{simplicity}: there are n edges meeting at each vertex;
    \item \textbf{rationality}: the edges meeting at the vertex p are rational in the sense that each edge is of the form $p+tu_i, t \geq 0$, with $u_i \in \mathbb{Z}^n$;
    \item \textbf{smoothness}: for each vertex, the corresponding $u_1,...,u_n$ can be chosen to be a $\mathbb{Z}$-basis of $\mathbb{Z}^n$.
\end{itemize}
\end{definition}

Delzant~\cite{De88} classified symplectic toric manifolds by their moment images.
\begin{theorem}[Delzant]
    There is a one-to-one correspondence bewteen the following two sets.
    \[\begin{tikzcd}
        \{\text{Symplectic Toric Manifolds}\}/{\sim} \arrow[r,leftrightarrow,"{1-1}"] & \{\text{Delzant Polytopes}\}/{\sim}
    \end{tikzcd}\]
    
The equivalence relation is given by equivariant symplectormorhism on the left and by translation on the right.
\end{theorem}

\begin{example}
Consider the symplectic manifold $(S^2, c_1(TS^2))$. In cylindrical coordinate, the symplectic form is $\frac{1}{2\pi} d\theta \wedge dh$. Let $S^1$ act on it by rotation: 
$$t.(\theta, h) = (\theta + 2\pi t, h).$$
The moment map is $h:S^2 \rightarrow \mathbb{R}$ and the moment polytope is $[-1,1]$.
\begin{figure}
    \centering
    \begin{tikzpicture}
  \draw (0,0) circle (2cm);
  \draw (-2,0) arc (180:360:2 and 0.6);
  \draw[dashed] (2,0) arc (0:180:2 and 0.6);
  \fill[fill=black] (0,2) circle (2pt);
  \fill[fill=black] (0,-2) circle (2pt);
  \draw[->] (4,0) -- (8,0);
  \draw (6,0) node[anchor = south] {$\Phi = h$};
  \draw (10,2) node[anchor = west] {1} -- (10,-2) node[anchor = west]{-1};
  \draw[dashed] (10,3) -- (10,2);
  \draw[dashed] (10,-3) -- (10,-2);
  \fill[fill=black] (10,2) circle (2pt);
  \fill[fill=black] (10,-2) circle (2pt);
\end{tikzpicture}
\end{figure}
\end{example}
\subsection{Sub-circle Actions}

Let $(M,\omega,T,\Phi)$ be a symplectic toric manifold with the Delzant polytope
\[\Delta = \bigcap_{i=1}^d \{ x \in \mathbb{R}^n \big| \, \langle x, v_i \rangle \leq \lambda_i \}.\]

We will give a formula for computing the stabilizer group under a sub-circle action of a non-fixed point whose moment image is on an edge of $\De$. We need the following lemma.

\begin{lemma}\label{lemma:stabilizer}
     Let $u, v_1, v_2,...,v_{n-1} \in \mathbb{Z}^n$ be linearly independent primitive vectors. Let $T' \subset T$ be the $(n-1)$ dimensional subtorus whose Lie algebra is generated by $v_1,v_2,...,v_{n-1}$ and $S^1_u \subset T$ be the circle whose Lie algebra is $\R u$. Let $P = \{tu - \sum_{i=1}^{n-1} t_i v_i \big| \, 0 \leq t, t_i \leq 1\}$. Then 
     \[|S^1_u \cap T'| = |\mathbb{Z}^n \cap \textrm{int}(P)| + 1 .\]
\end{lemma}

\begin{proof}
    Write $u = (u^1,...,u^n), v_i= (v_i^1,...,v_i^n)$. Then \[S^1_u= \{[tu^1,...,tu^n] \in \R^n/\Z^n \big|  \, t \in [0,1) \}, \] and \[T' = \big \{ [\sum_{j=1}^{n-1} t_j v_j^1,...,\sum_{j=1}^{n-1} t_j v_j^n] \in \R^n/\Z^n \big|  \, t_j \in [0,1) \big \}. \] 
    
    Thus, there is a one-to-one correspondence between $(S^1_u \cap T') \setminus \{[0,...,0]\}$ and 
    \[Q:= \{(t,t_1,...,t_{n-1}) \in (0,1)^n \big| \, tu^k \equiv \sum_{j=1}^{n-1} t_j v_j^k \pmod{\Z} \textrm{ for all } k = 1,...,n\}. \]

    Now, we show that the matrix $A = (u,-v_1,...,-v_{n-1})$ maps $Q$ bijectively to $\mathbb{Z}^n \cap \textrm{int}(P)$.
    
    For any $(t,t_1,...,t_{n-1}) \in Q$, there exist unique $k_1,...,k_n \in \Z$ such that
    \begin{equation}\label{equation:stabilizer}
        \begin{cases}
    tu^1 - \sum_{j=1}^{n-1} t_j v_j^1= k_1\\
    tu^2 - \sum_{j=1}^{n-1} t_j v_j^2 = k_2\\
         \cdots \cdots \cdots \cdots \cdots \cdots \cdots\\
    tu^n - \sum_{j=1}^{n-1} t_j v_j^n = k_n
    \end{cases}
    \end{equation}
    
    Thus, $(k_1,...,k_n)^T = A(t,t_1,...,t_{n-1})^T \in \mathbb{Z}^n \cap \textrm{int}(P)$. 

    Since $u ,v_1,...,v_{n-1}$ are linearly independent, $A$ is invertible. For any $(k_1,...,k_n)^T \in \mathbb{Z}^n \cap \textrm{int}(P)$, there exist unique $(t,t_1,...,t_{n-1})^T = A^{-1}(k_1,...,k_n)^T \in (0,1)^n$. Now, equation~(\ref{equation:stabilizer}) implies that $(t,t_1,...,t_{n-1}) \in Q$, i.e. $\mathbb{Z}^n \cap \textrm{int}(P) = A Q$.

    Since $A \in \GL(n,\Z)$, we conclude that $|S^1_u \cap T'| = |Q|+1 = |\mathbb{Z}^n \cap \textrm{int}(P)| + 1 $.
\end{proof}

Now we apply this lemma to the sub-circle action.

Recall that for an edge $E$ of $\De$, there exists a unique $J_E \subset \{1,2,....,d\}$ with $|J_E| = n-1$ such that
\[E = \Delta \cap \bigcap_{j \in J_E}  \{ x \in \mathbb{R}^n \big| \, \langle x, v_j \rangle = \lambda_j \}.  \]

\begin{corollary}\label{cor:stabilizer}
Let $u \in \Z^n$ be a primitive vector and $S^1_u \subset T$ be the circle whose Lie algebra is $\R u$. Let $P^u_E =\{tu-\sum_{j \in J_E} t_jv_j \big| \, 0\leq t,t_j \leq 1\}$, $k^u_E = |\Z^n \cap \textrm{int}(P^u_E)|+1$, and $m_u(\De) = \max_E \{k^u_E \big| \, E \textrm{ is an edge of } \De\}$. Then
\[m_u(\De) = \max_p \{ \, |\textrm{stab}_p|:\, p \in M \setminus M^{S^1_u}\}.\]
\end{corollary}
\begin{proof}
    By Delzant's construction\cite{De88}, for any $x \in M \setminus M^T$, there exists an edge $E \subset \De$ and a $p \in M\setminus M^T$ such that $\Phi(p) \in E$ and $\textrm{stab}_x \subset \textrm{stab}_p$. Furthermore, $\textrm{stab}_p$ for the $T$-action is an $(n-1)$ dimensional subtorus $T_E$ whose Lie algebra is generated by $\{v_j:j \in J_E\}$. Therefore, for the $S^1_u$-action,
    \[\max_p\{ \, |\textrm{stab}_p|:\, p \in M \setminus M^{S^1_u}\} = \max_p\{ \, |\textrm{stab}_p|:\, p \in M \setminus M^{S^1_u}, \Phi(p) \in E\}.\]
    
    Let $p \in M\setminus M^T$ such that $\Phi(p) \in E$ for some edge $E\subset \De$. The stabilizer group of $p$ under the $S^1_u$-action is $S^1_u \cap T_E$. If $u \in \textrm{Span}\{v_j: j \in J_E\}$, then $S^1_u \subset T_E$, so $p \in M^{S^1_u}$. In this case, $ \textrm{int}(P^u_E) = \emptyset$, so $k^u_E=1$. If $u \not \in \textrm{Span}\{v_j: j \in J_E\}$, by Lemma~\ref{lemma:stabilizer}, $\textrm{stab}_p = S^1_u \cap T_E$ has cardinality $k^u_E$. Since $k^u_E \geq 1$,
    \[\max_p\{ \, |\textrm{stab}_p|:\, p \in M \setminus M^{S^1_u}, \Phi(p) \in E\} = \max_E\{k^u_E \big| \, E \textrm{ is an edge of } \De\} = m_u(\De).\]
\end{proof}

\section{The Hofer-Zehnder Capacity}\label{s3}
In this section, we review the definition of the Hofer-Zehnder capacity and prove the key lemma used in the proof of the main theorem. Recall that $H: M \to \R$ is \textbf{admissible} if 
\begin{itemize}
    \item there exist open sets $U,V$ such that $H|_U = H_{\max}$ and $H|_V = H_{\min}$, and
    \item $X_H$ has no non-constant periodic orbits of period less than $1$.
\end{itemize}
We denote the set of admissible functions by $\mathcal{H}_{ad}(M,\omega) = \{ H \in C^\infty(M)$ \big| $H$ is admissible$\}$. 

\begin{definition}\label{def:HZ}
    The \textbf{Hofer-Zehnder capacity} of $(M,\omega)$ is defined by \[c_{\textrm{HZ}}(M,\omega) = \textrm{sup} \{H_{\max} -H_{\min} \big| \, H \in \mathcal{H}_{ad}(M, \omega) \}. \]
\end{definition}

An equivalent description of the Hofer-Zehnder capacity is given below.

For a smooth function $H: M \to \R$, let  \[\mathcal{T}_H:= \textrm{inf} \{\mathcal{T} \big| \, \mathcal{T} \textrm{ is the period of a non-constant periodic orbit of } X_H\}.\]

\begin{proposition}\label{prop:equiv-def}
    Let $\tilde{\mathcal{H}}(M) := \{ H \in C^{\infty}(M)$ \big| $H(M) = [0,1]$ , $H$ attains $0,1$ on open sets$\}$. Then 
 \[c_{\textrm{HZ}}(M,\omega) = \textrm{sup} \{ \mathcal{T}_{H} \big| \, H \in \tilde{\mathcal{H}}(M)\}.\]
\end{proposition}

We first establish a lemma that will be used in the proof of both Proposition~\ref{prop:equiv-def} and a later result.

\begin{lemma}\label{lem:period}
    Let $(M,\omega)$ be a symplectic manifold and $f: M \to \R$ be a non-constant smooth function. Then for any $a \in \R>0$, let $g = af$, then $\mathcal{T}_{g} = \frac{\mathcal{T}_{f}}{a}$.
\end{lemma}
\begin{proof}
    Since $dg =adf$, $X_{g} = aX_f$. Let $\gamma(t)$ be an integral curve of $X_f$ starting from a point $p \in M$. Let $\eta(t) := \gamma(at)$, then \[\eta'(t) = a\gamma'(at) = aX_f(\gamma(at)) = aX_f(\eta(t)) = X_{g}(\eta(t)).\]
    The flow $\phi^t_f$ for $X_f$ and the flow $\phi^t_g$ for $X_g$ are related by $\phi^t_g = \phi^{at}_f.$
    Therefore, there is a one-to-one correspondence between the non-constant periodic orbits of $X_f$ and the non-constant periodic orbits of $X_g$. In addition, for any non-constant periodic orbits of $X_f$ with period $\mathcal{T}$, the corresponding periodic orbits of $X_g$ has period $\frac{\mathcal{T}}{a}$. The lemma follows.
\end{proof}

\begin{proof}[Proof of Proposition~\ref{prop:equiv-def}]
    For any non-constant function $ H \in \mathcal{H}_{ad}(M, \omega)$, let $||H||:= H_{\max} -H_{\min}$, then \[\tilde{H}: = \frac{H-H_{\min}}{||H||} \in \tilde{\mathcal{H}}(M).\]
By Lemma~\ref{lem:period}, $\mathcal{T}_{\tilde{H}}  \geq ||H|| \mathcal{T}_H$. Since $H$ is admissible, $\mathcal{T}_H \geq 1$, so $\mathcal{T}_{\tilde{H}} \geq ||H||$.

Taking supremum over $\mathcal{H}_{ad}(M, \omega)$,
\[\textrm{sup} \{\mathcal{T}_{\tilde H}\big| \, H \in \mathcal{H}_{ad}(M, \omega)\} \geq \textrm{sup} \{H_{\max} -H_{\min} \big| \, H \in \mathcal{H}_{ad}(M, \omega)\} =  c_{\textrm{HZ}}(M,\omega).\]

Since $\{\mathcal{T}_{\tilde H} \big| \, H \in \mathcal{H}_{ad}(M, \omega)\} \subset \{ \mathcal{T}_{H} \big| \, H \in \tilde{\mathcal{H}}(M)\}$, we conclude that
\[\textrm{sup} \{\mathcal{T}_{H} \big| \, H \in \tilde{\mathcal{H}}(M)\} \geq \textrm{sup} \{\mathcal{T}_{\tilde H}\big| \, H \in \mathcal{H}_{ad}(M, \omega)\} \geq  c_{\textrm{HZ}}(M,\omega).\]
Conversely, for any $\tilde{H} \in \tilde{\mathcal{H}}(M)$ such that $\mathcal{T}_{\tilde{H}} >0$, let $H:= \mathcal{T}_{\tilde{H}}  \tilde{H}$. By Lemma~\ref{lem:period}, $\mathcal{T}_{H} \geq  \frac{\mathcal{T}_{\tilde{H}}}{\mathcal{T}_{\tilde{H}}} =1.$
Thus, $H\in \mathcal{H}_{ad}(M, \omega)$ and $||H || = \mathcal{T}_{\tilde{H}}$.  This implies that
\[\{\mathcal{T}_{\tilde{H}}:\tilde{H} \in \tilde{\mathcal{H}}(M)\} \subset \{H_{\max} -H_{\min} \big| \, H \in \mathcal{H}_{ad}(M, \omega) \}.\]

Therefore, \[\textrm{sup} \{\mathcal{T}_{\tilde{H}} \big| \, \tilde{H} \in \tilde{\mathcal{H}}(M)\} \leq  c_{\textrm{HZ}}(M,\omega).\]

We conclude that  \[c_{\textrm{HZ}}(M,\omega) = \textrm{sup} \{ \mathcal{T}_{H} \big| \, H \in \tilde{\mathcal{H}}(M)\}.\]
\end{proof} 

For a Hamiltonian circle action on a symplectic manifold, its moment map does not attain its max and min on an open set, so it is not an admissible function. We now show that we can always modify the moment map of a Hamiltonian $S^1$-action to an admissible function while keeping track of its magnitude, i.e. the difference between its max and its min. In \cite{HS17}, the authors modified the moment map of a semi-free Hamiltonian $S^1$-action to get an admissible function. Specifically, they modified the function locally in an $S^1$-invariant neighborhood of a minimal fixed point. The similar technique will not work here, as their modification is based on the fact that near a minimal fixed point, the isotropy weights are all $1$. Instead, Lemma 4.0.1 in \cite{Bimmermann} does the work. We remark that in our proof we need to first divide the moment map by a constant to ensure that the non-constant periodic orbits have periods at least $1$.

\begin{lemma}\label{lemma:modification}
    Let $(M,\omega)$ be a compact, connected symplectic manifold. Let $S^1$ act effectively on $(M,\omega)$ with a moment map $\Phi:M \to \R$. Let $\mathcal{T} = \max\{ \, |\textrm{stab}_p|:\, p \in M \setminus M^{S^1}\}$. Then there exists an admissible function $H$ such that $H_{\max} -H_{\min} = \frac{1}{\mathcal{T}}(\Phi_{\max} -\Phi_{\min})$
\end{lemma}
\begin{proof}
    Let $\Psi = \frac{\Phi}{\mathcal{T}}$. We first show that the non-constant periodic orbits of $X_\Psi$ are of period at least $1$. By Lemma~\ref{lem:period}, it suffices to show that the non-constant periodic orbits of $X_\Phi$ are of period at least $\frac{1}{\mathcal{T}}$.

    For any free orbit, its period is $1$ which is greater than or equal to $\frac{1}{\mathcal{T}} \geq 1$.

    Suppose $\mathcal{O}_p$ is a non-free orbit with $\textrm{stab}_p \cong \Z/k\Z$ for some $k > 1$. Then its period is $\frac{1}{k}$. Since $\mathcal{T} \geq k$, we conclude that $\frac{1}{k} \geq \frac{1}{\mathcal{T}}$.

    Now, we modify $\Psi$ so that it achieves its maximum (and its minimum, respectively) on an open set. Let $f: [\Psi_{\min},\Psi_{\max}] \to [0,\infty)$ be a smooth function that satisfies the following properties:
    \begin{itemize}
        \item[(i)] there exists an $\epsilon > 0$ such that $f(x) = 0$ for $x \in [\Psi_{\min},\Psi_{\min}+\epsilon)$ and $f(x) =\Psi_{\max}-\Psi_{\min}$ for $x \in (\Psi_{\max}-\epsilon,\Psi_{\max}]$.
        \item[(ii)] $0 \leq f'(x) < 1$.
    \end{itemize}

    Let $H = f \circ \Psi$. By item (i), $H$ attains its maximum and minimum on open sets. Since each non-constant periodic orbit of $\Psi$ has period at least $1$, by item (ii), each non-constant periodic orbit of $X_H$ has period at least $\frac{1}{f'(x)} \geq 1$. Hence, $H$ is an admissible function and $H_{\max} -H_{\min} = \Psi_{\max}-\Psi_{\min} =  \frac{1}{\mathcal{T}}(\Phi_{\max} -\Phi_{\min})$.
\end{proof}

\section{Examples}\label{s4}

In this section, we describe an explicit algorithm to find the lower bound of the Hofer-Zehnder capacity provided in Theorem~\ref{thm:main} for $4$-dimensional symplectic toric manifolds. The technical difficulty is to find an explicit formula for the number of lattice points in the interior of a simple lattice polytope. The computation in dimension 4 is feasible thanks to the Pick's Theorem.

\begin{theorem}[Pick's Theorem]\label{thm:pick}
Let $\Delta \subset \mathbb{R}^2$ be a convex lattice polygon. Let $A$ denote the area of $\Delta$, $i$ denote the number of lattice points in the interior of the polygon, $b$ denote the number of lattice points on the boundary. Then
\[A = i + \frac{b}{2} - 1.\]
\end{theorem}

\begin{corollary}\label{cor:Det}
    Let $\De \subset \R^2$ be a Delzant polygon. Let $E$ be an edge of $\De$ with the primitive outward normal vector $v$. For any primitive vector $u \in \Z^2 \setminus\{\pm v\}$, $k^u_E = |\det(u,-v)|$.
    In particular, if $v_1,\ldots,v_d$ are primitive outward normal vectors of edges of $\De$, then 
    \[m_u(\De) = \max\{|\det(u,-v_1)|,...,|\det(u,-v_d)|\}.\]
\end{corollary}
\begin{proof}
    $P^u_E = \{tu-sv: t,s \in [0,1]\}$ is a parallelogram whose vertices are lattice points. By Theorem~\ref{thm:pick}, $k^u_E = |\Z^n \cap \textrm{int}(P^u_E)|+1 = i+1 = A-\frac{b}{2}+2$. Since $u,v$ are primitive, there are four lattice points on the boundary of $P^u_E$, namely $0,u,-v,u-v$. Hence, $b = 4$. On the other hand, $A= |\det(u,-v)|$, so \[k^u_E = |\det(u,-v)|- \frac{4}{2}+2 =|\det(u,-v)|. \]

    If $u,v_i$ are linearly dependent, then there exists some $j$ such that $u,v_j$ are linearly independent. Hence, $k^u_{E_j} = |\det(u,-v_j)| \geq 1 = k^u_{E_i}$, which implies \[\max\{k^u_{E_1},\ldots,k^u_{E_d}\} = \max\{|\det(u,-v_1)|,...,|\det(u,-v_d)|\}.\]
\end{proof}

\begin{example}\label{eg:P2}
    Let $c>0$. Consider the symplectic toric manifold $(\mathbb{CP}^2, c\omega, T^2,\Phi)$ whose moment image $\De$ is given in Figure~\ref{fig:P2}:

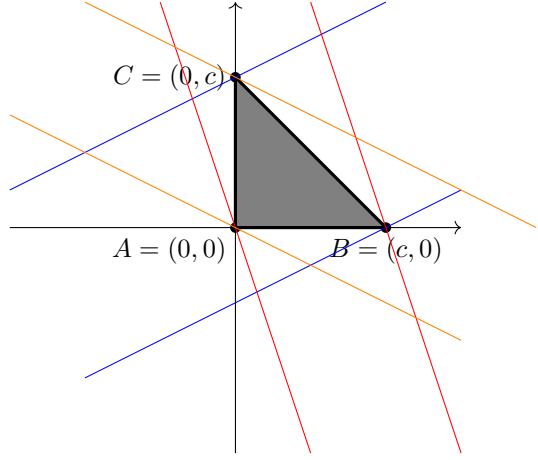
\begin{figure}
    \centering
    \begin{tikzpicture}
      \fill[fill=black] (0,0) node[below left] {$A= (0,0)$} circle (2pt);
      \draw[->] (0,-3) -- (0,0) -- (0,3);
      \draw[->] (-3,0) -- (0,0) -- (3,0);
      \fill[fill=black] (2,0) node[anchor = north] {$B=(c,0)$}  circle (2pt);
      \fill[fill=black] (0,2) node[anchor = east] {$C=(0,c)$} circle (2pt);
      \filldraw[color= black, fill= gray, very thick] (2,0) -- (0,0) -- (0,2) -- (2,0);
      \draw[blue] (-3,0.5) -- (0,2) -- (2,3);
      \draw[blue] (-2,-2) -- (2,0) -- (3,0.5);
      \draw[red]  (3,-3) -- (2,0) -- (1,3);
      \draw[red]  (1,-3) -- (0,0) -- (-1,3);
      \draw[orange] (-2,3) -- (0,2) -- (4,0);
      \draw[orange] (-3,1.5) -- (0,0) -- (3,-1.5);
    \end{tikzpicture}
    \caption{The moment image of $\C\P^2$.}
    \label{fig:P2}
\end{figure}

We denote the primitive outward normal vectors by $v_1 = (0,-1), v_2 = (1,1), v_3 = (-1,0)$. Let $u = (p,q) \in \Z^2 \setminus\{0\}$ be a primitive vector. $|\det (u, -v_1)| = |p|$, $|\det (u, -v_2)| = |p-q|, |\det (u, -v_3)|  = |q|$. By Corollary~\ref{cor:Det}, $m_u(\De) = \max\{|p|,|p-q|,|q|\}$.

{\bf Case 1.} $pq <0$. Then we can assume without loss of generality that $p>0,q<0$, since multiplying $u$ by $-1$ will not change $m_u(\De)$. Then $m_u(\De) = p-q$. \[\mathcal{T}_u(\De) = \frac{1}{p-q}\big(\smash{\max_{(x,y) \in \Delta}} (px+qy) - \smash{\min_{(x,y) \in \Delta}} (px+qy)\big).\] 

If we think of the geometric meaning of $px+qy$ and write $b=px+qy$, then $b$ is maximized when the y-intercept of the line $y = -\frac{p}{q}x+y_0$ is minimized.

As the blue line in the picture shows, $b$ is maximized at $B$ and minimized at $C$. So $\mathcal{T}_u(\De) = \frac{1}{p-q}(cp-cq) = c$.

{\bf Case 2.} $pq > 0$. Then we can assume without loss of generality that $p,q>0$.

In this case, $m_u(\De) = \max(p,q)$. 

We want to maximize and minimize $px+qy$ on $\Delta$. Similarly, when $p>q$, as the red line shows, it's maximized at $B$ and minimized at $A$. So $\mathcal{T}_u(\De) = \frac{1}{p}(cp-0) = c$. The same argument works for $p<q$ (the orange line) and we still get $\mathcal{T}_u(\De) = c$.

{\bf Case 3.} $u=(0,\pm 1)$ or $u= (\pm 1,0)$. It is straightforward to check that $\mathcal{T}_u(\De) =c$.

By Theorem~\ref{thm:main}, we get the estimate $c_{\textrm{HZ}}(\mathbb{CP}^2,c\omega) \geq c$. 
\end{example}

\begin{example}\label{eg:P1P1}
For $a,b \in \R_{>0}$, consider $(S^2 \times S^2, a\, \omega \oplus b\, \omega, T^2)$ whose moment image is $\De = [0,a] \times [0,b]$, where $\omega=\frac{1}{2}c_1(TS^2)$. We denote the primitive outward normal vectors by $v_1 = (-1,0), v_2 = (0,-1), v_3 = (1,0), v_4 = (0,1)$.

For any primitive vector $u = (p,q) \in \Z^2 \setminus\{0\}$, by Corollary~\ref{cor:Det}, $m_u(\De) = \max\{|p|,|q|\}$.

A computation similar to the previous example shows that 
\[\mathcal{T}_u(\De) = \frac{a|p|+b|q|}{\max\{|p|,|q|\}}.\]

Notice that $\mathcal{T}_u(\De) \leq a+b$ for all primitive vectors $u$. Furthermore, for $u = (1,1)$, $\mathcal{T}_u(\De) = a+b$. Hence, $w_T(\De)=a+b$ and $c_{\textrm{HZ}}(S^2 \times S^2, a\, \omega \oplus b\, \omega) \geq a+b$. 
\end{example}

\begin{remark}
    In \cite{HS17}, the authors proved that if $(M,\omega)$ is a closed Fano symplectic manifold with a semi-free $S^1$-action and the Hamiltonian function $H:M \to \R$ attains its max at a single point, then $c_{\textrm{HZ}}(M,\omega) = H_{\max} - H_{\min}$. In particular, they showed that $c_{\textrm{HZ}}(\mathbb{CP}^2,3\omega) =3$ and $c_{\textrm{HZ}}(S^2 \times S^2, a \omega \oplus b \omega) =a+b$. By the conformality of the symplectic capacity, we further obtain that $c_{\textrm{HZ}}(\mathbb{CP}^2,c\omega) =c$ for any $c>0$. Hence, our estimate is sharp in these cases. In fact, our estimate is always sharp in the Fano case.
\end{remark}

\begin{example}
    Given $n \in \Z_{>0}$, let $\mathcal{O}(n) \to \C\P^1$ be the holomorphic line bundle whose first Chern class is $n$. Consider the $n$-th Hirzebruch surface $\mathcal{H}_n : = \P(\mathcal{O} \oplus \mathcal{O}(-n))$. Under a suitable choice of the torus action and the symplectic form, its moment image $\De$ is the trapezoid shown in Figure~\ref{fig:Hn}, where $a,b >0$. The primitive outward normal vectors are $v_1 = (0,-1), v_2 = (1,n), v_3 = (0,1), v_4 =(-1,0)$. 
    
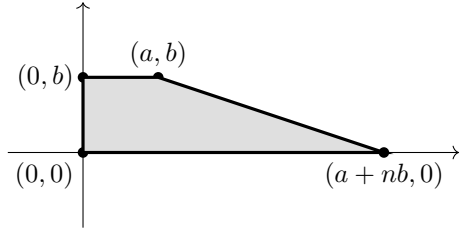
\begin{figure}
    \centering
    \begin{tikzpicture}
      \fill[fill=black] (0,0) node[below left] {$(0,0)$} circle (2pt);
      \draw[->] (0,-1) -- (0,0) -- (0,2);
      \draw[->] (-1,0) -- (0,0) -- (5,0);
      \fill[fill=black] (4,0) node[anchor = north] {$(a+nb,0)$}  circle (2pt);
      \fill[fill=black] (0,1) node[anchor = east] {$(0,b)$}  circle (2pt);
      \fill[fill=black] (1,1) node[anchor = south] {$(a,b)$} circle (2pt);
      \filldraw[color= black, fill= gray!25, very thick] (0,1) -- (1,1) -- (4,0) --  (0,0) -- (0,1);
    \end{tikzpicture}
    \caption{The moment image of the $n$-th Hirzebruch surface}
    \label{fig:Hn}
\end{figure}
    
    For any primitive vector $u = (p,q) \in \Z^2 \setminus \{0\}$, by Corollary~\ref{cor:Det}, $m_u(\De)=\max\{|p|,|np-q|,|q|\}$. 

    When $n = 1$, $m_u(\De)=\max\{|p|,|p-q|,|q|\}$. A discussion similar to Example~\ref{eg:P2} shows that $w_T(\De) = a+b$. Hence, $ c_{\textrm{HZ}}(\mathcal{H}_1,\omega_{\Delta}) \geq a+b$. For example, when $(a,b) = (1,2)$, this toric manifold is Fano and the estimate is sharp.

    For $n \geq 2$, we put the data in the table below.

\begin{table}[h!]
  \centering
  \begin{tabular}{|c|c|c|}
    \hline
    $u$   & $m_u(\De)$ & $\mathcal{T}_u(\De)$ \\[0.5ex] 
    \hline
    $(0,\pm 1)$ & $1$ & $b$\\
    \hline
    $(\pm 1,0)$ & $n$ & $\frac{a+nb}{n}$\\
    \hline
    $\frac{p}{q}< 0$ & $|np-q|$ & $\frac{|pa+(np-q)b|}{|np-q|}$\\ 
    \hline
    $0 <\frac{p}{q} \leq \frac{1}{n}$ & $|q|$ & $\frac{|p|a+|q|b}{|q|}$\\
    \hline
    $\frac{1}{n} <\frac{p}{q} \leq \frac{2}{n}$ & $|q|$ & $\frac{(a+nb)|p|}{|q|}$\\
    \hline
    $\frac{p}{q} > \frac{2}{n}$ & $|np-q|$ & $\frac{(a+nb)|p|}{|np-q|}$\\
    \hline
  \end{tabular}
  \label{tab:table}
\end{table}
When $\frac{p}{q} < 0$, $$\frac{|pa+(np-q)b|}{|np-q|} \leq  \frac{|p|a}{|np-q|} + b =  \frac{a}{|n-\frac{q}{p}|} +b < \frac{a}{n} + b$$

When $0 <\frac{p}{q} \leq \frac{1}{n}$, $$\frac{|p|a+|q|b}{|q|}  = a \frac{p}{q}+b \leq \frac{a}{n} +b$$

When $\frac{1}{n} <\frac{p}{q} \leq \frac{2}{n}$, \begin{equation}\label{eqn:Hn}
    \frac{(a+nb)|p|}{|q|} \leq \frac{2a}{n}+2b
\end{equation}

When $\frac{p}{q} > \frac{2}{n}$,  $$\frac{(a+nb)|p|}{|np-q|} = \frac{a+nb}{|n -\frac{q}{p}|} < \frac{2}{n}(a+nb) =\frac{2a}{n}+2b$$

Furthermore, when $(p,q) = \left(\frac{2}{\gcd(2,n)}, \frac{n}{\gcd(2,n)}\right)$, the inequality~(\ref{eqn:Hn}) is an equality. Hence, $w_T(\De) = \frac{2a}{n}+2b$. 

Therefore, for $n \geq 2$, $c_{\textrm{HZ}}(\mathcal{H}_n,\omega_{\Delta}) \geq \frac{2a}{n}+2b$. $c_{\textrm{HZ}}(\mathcal{H}_n,\omega_{\Delta})$ is not known in this case.

\end{example}

\begin{remark}
    Though the computation is straightforward in dimension $2$ (i.e. when we have symplectic toric $4$-manifolds), it is not known whether there is a closed formula for the number of interior lattice points of a simple lattice polytope in higher dimension. Brion and Vergne~\cite{BV97} gave a formula in terms of the derivative of the Todd operator. It is not clear to us how this formula could be applied to extract the exact numbers.
\end{remark}

\bibliographystyle{amsplain}
\bibliography{ref}

\end{document}